\documentclass[11pt]{article}
\usepackage[T1]{fontenc}
\usepackage[utf8]{inputenc}
\usepackage{amsmath,amsfonts,amssymb,amsthm,mathtools}
\usepackage{enumitem}
\usepackage[margin=1in]{geometry}
\usepackage{microtype}
\usepackage[hidelinks]{hyperref}
\DeclareMathOperator{\spn}{span}
\newcommand{\F}{\mathbb F_2}
\newcommand{\E}{\mathbb E}
\newcommand{\Prob}{\mathbb P}
\newcommand{\Cr}{C_r^*(\F)}
\newcommand{\M}{L(\F)}
\newcommand{\norm}[2]{\left\lVert #1\right\rVert_{#2}}
\newtheorem{theorem}{Theorem}[section]
\newtheorem{proposition}[theorem]{Proposition}
\newtheorem{lemma}[theorem]{Lemma}
\newtheorem{corollary}[theorem]{Corollary}
\theoremstyle{remark}
\newtheorem{remark}[theorem]{Remark}
\numberwithin{equation}{section}
\title{Fourier rearrangements on discrete groups:\\$L^p$ and operator-norm subsequence convergence}
\author{Tao Mei and Sebastian Vargas-Loaiza}
\date{September 14, 2026}

\begin{document}
\maketitle

\begin{abstract}
Let $\Gamma$ be an infinite finitely generated group that is hyperbolic
or of polynomial growth.  For $f$ belonging to the reduced group $C^*$ algebra $C_r^*(\Gamma)$, we construct a rearrangement of its
individual Fourier terms with a subsequence converging to $f$ in the operator norm.
For every $2\le p<\infty$ and
$f$ belonging to the noncommutative $L^p$ space  $ L^p(\hat \Gamma)$, 
we obtain convergence in the $L^p$
norm.  The proof combines uniformly bounded
finite-support Fourier cutoffs with an elementary Bernoulli
even-moment estimate.   We give a self-contained free-group
argument, then establish a general criterion and verify it for the
two classes above. The operator-norm conclusion establishes the
R\'ev\'esz property for these groups. In particular, it confirms the
polynomial-growth conjecture of Hamm, Hayes, and Petrosyan and disproves
their conjecture that the free group $\F$ does not have this property.
\end{abstract}

\section{Introduction and main results}
Let $\F$ be the free group on two generators, let $|g|$ denote reduced
word length, and let $\lambda_g\delta_h=\delta_{gh}$ be its left regular
representation. Its reduced $C^*$-algebra is
\[
 \Cr=\overline{\spn\{\lambda_g:g\in\F\}}^{\,\norm{\cdot}{\infty}}.
\]
The canonical trace on the group von Neumann algebra $\M$ is normalized
by $\tau(1)=1$ and satisfies $\tau(\lambda_g)=\mathbf1_{\{g=e\}}$.
Inner products are taken conjugate-linear in the first variable.
We use the notation
$L^p(\widehat{\F})=L^p(\M,\tau)$, with
$\norm{x}{p}=\tau(|x|^p)^{1/p}$ for $p<\infty$, and
$L^\infty(\widehat{\F})=\M$. For $f\in L^2(\M,\tau)$, write
\[
 c_g=\widehat f(g)=\tau(\lambda_g^*f),\qquad
 f\sim\sum_{g\in\F}c_g\lambda_g.
\]

The metric approximation property of reduced free group $C^*$-algebras
was proved by Haagerup \cite{Ha79}. Nevertheless, convergence of the
unweighted Fourier partial sums of a given element remains delicate.
Bo\.{z}ejko and Fendler \cite{BF06} showed that, for $p>3$, the natural
Dirichlet-type partial-sum operators do not yield norm convergence for
every element of the free group noncommutative $L^p$ space. The
operator-norm problem is subtler still. Motivated by R\'ev\'esz's
classical theorem \cite{Re90}, one may ask whether, for each
$f\in C_r^*(\F)$, the free group elements can be enumerated so that a
subsequence of the corresponding Fourier partial sums converges to $f$.
We answer this question affirmatively. This also disproves a conjecture
of Hamm, Hayes, and Petrosyan \cite{HHP19}.
Beyond its intrinsic interest, the convergence of Fourier partial sums has also arisen in the recent work of Andreou, in which the author requires only a weaker condition called the $D_\infty$-property; see Proposition 5.7 in \cite{An23}. 

All rearrangements below keep these coefficients unchanged.

\begin{theorem}\label{thm:main}
Fix $2\le p<\infty$. For every $f\in L^p(\widehat{\F})$, there
exist an enumeration $(g_\ell)_{\ell\ge1}$
of $\F$ and strictly increasing positive integers $M_j$ such that
\begin{equation}\label{eq:main}
 \lim_{j\to\infty}\norm{\sum_{\ell=1}^{M_j}
              \widehat f(g_\ell)\lambda_{g_\ell}-f}{p}=0.
\end{equation}
The same subsequence converges in every $L^q$, $1\le q\le p$.
More precisely, there are integers $n_{j+1}>2n_j$ and finite sets
$E_j\subset\F$ such that
\begin{equation}\label{eq:sets}
 \{|g|\le n_j\}\subset E_j\subset\{|g|\le2n_j\},\qquad
 E_j\subsetneq E_{j+1},\qquad \bigcup_jE_j=\F,
\end{equation}
and
\begin{equation}\label{eq:rate}
 \norm{\sum_{g\in E_j}\widehat f(g)\lambda_g-f}{p}
 \le 2^{1-j}.
\end{equation}
The enumeration can be chosen so that its first $M_j=|E_j|$ elements
are exactly $E_j$.
\end{theorem}

\begin{theorem}[Operator-norm convergence]\label{thm:operator}
For every $f\in\Cr$, there exist an enumeration
$(g_\ell)_{\ell\ge1}$ of $\F$ and strictly increasing positive
integers $M_j$ such that
\begin{equation}\label{eq:operator}
 \lim_{j\to\infty}\norm{\sum_{\ell=1}^{M_j}
              \widehat f(g_\ell)\lambda_{g_\ell}-f}{\infty}=0.
\end{equation}
The selected sets $E_j=\{g_1,\ldots,g_{M_j}\}$ can satisfy
\eqref{eq:sets} and
\begin{equation}\label{eq:operatorrate}
 \norm{\sum_{g\in E_j}\widehat f(g)\lambda_g-f}{\infty}
 \le2^{1-j}.
\end{equation}
The same subsequence then converges in every finite $L^q$,
$1\le q<\infty$.
\end{theorem}

\begin{corollary}[Simultaneous convergence]\label{cor:simultaneous}
If $f\in\bigcap_{2\le q<\infty}L^q(\widehat{\F})$, there are one
enumeration and one subsequence whose partial sums converge to $f$
in every finite $L^q$, $1\le q<\infty$. This holds, in particular,
for every $f\in\M$. The selected sets may satisfy \eqref{eq:sets} and
\begin{equation}\label{eq:simrate}
 \norm{\sum_{g\in E_j}\widehat f(g)\lambda_g-f}{2j}\le2^{1-j}.
\end{equation}
\end{corollary}

The free-group results extend to the following classes. For a discrete
group $\Gamma$ we write $L^p(\widehat\Gamma)=L^p(L(\Gamma),\tau)$,
use its canonical trace, and define Fourier coefficients as above.
For a fixed finite symmetric generating set, let
$B_N^\Gamma=\{g\in\Gamma:|g|\le N\}$.

\begin{theorem}[Hyperbolic groups and polynomial growth]\label{thm:groups}
Let $\Gamma$ be an infinite finitely generated group that is either
word-hyperbolic or of polynomial growth. Then:
\begin{enumerate}[label=\textup{(\roman*)},leftmargin=*]
\item For each $2\le p<\infty$ and $f\in L^p(\widehat\Gamma)$,
there are an enumeration $(g_\ell)_{\ell\ge1}$ of $\Gamma$ and
strictly increasing integers $M_j$ such that
\[
 \norm{\sum_{\ell=1}^{M_j}\widehat f(g_\ell)\lambda_{g_\ell}-f}{p}
 \longrightarrow0.
\]
The same subsequence converges in every $L^q$, $1\le q\le p$.
\item For each $f\in C_r^*(\Gamma)$, an enumeration and a
subsequence can be chosen with
\[
 \norm{\sum_{\ell=1}^{M_j}\widehat f(g_\ell)\lambda_{g_\ell}-f}{\infty}
 \longrightarrow0.
\]
\end{enumerate}
In either case the selected sets $E_j=\{g_1,\ldots,g_{M_j}\}$ may
be chosen with error at most $2^{1-j}$ in the asserted norm and with
\begin{equation}\label{eq:groupssets}
 B_{n_j}^\Gamma\subset E_j\subset B_{a n_j}^\Gamma,\qquad
 n_{j+1}>a n_j,\qquad E_j\subsetneq E_{j+1},\qquad
 \bigcup_j E_j=\Gamma.
\end{equation}
Here $a=2$ is available in the hyperbolic case and $a=3$ in the
polynomial-growth case. The simultaneous-convergence assertion of
Corollary~\ref{cor:simultaneous} also holds with $\Gamma$ in place
of $\F$; in particular it holds for every $f\in L(\Gamma)$.
\end{theorem}

For a finite group, the corresponding conclusions hold simply by
taking the full finite Fourier sum, which equals $f$.

Hamm, Hayes, and Petrosyan \cite[Definition~4]{HHP19} say that a
countable discrete group has the \emph{R\'ev\'esz property} if every
$f\in C_r^*(\Gamma)$ admits a nested exhaustion by finite sets $E_j$
whose Fourier sums converge to $f$ in operator norm. For an infinite
group this is equivalent to the rearranged-subsequence formulation:
discard repeated sets and enumerate their successive differences.
In Section~III of their 2019 paper, immediately
after Theorem~III.1, they conjecture that $\F$ does not have this
property and that groups of polynomial growth do have it.

\begin{corollary}[The Hamm--Hayes--Petrosyan conjectures]\label{cor:conjectures}
Every finitely generated hyperbolic group and every finitely generated
group of polynomial growth has the R\'ev\'esz property. In particular,
$\F$ has the property. Thus the first conjecture in
\cite[Section~III]{HHP19} is false and the second is true.
\end{corollary}
\begin{proof}
Apply Theorem~\ref{thm:groups}\textup{(ii)}, with the preceding
observation for finite groups.
\end{proof}

We first prove the free-group statements directly. In this case,
the cutoffs used in the proof keep every coefficient of length at most
$N$, taper the coefficients between $N$ and $2N$, and vanish beyond
$2N$. Uniform boundedness and density give convergence of these
weighted sums in finite $L^p$, and in operator norm on $\Cr$.
We then replace the weights on each transition annulus by independent
$0$--$1$ choices. For finite $p$, the rounding error is controlled
directly in $L^p$. For the operator norm, we apply Haagerup's inequality
to powers of the rounding error and take the moment exponent
proportional to the logarithm of its degree. A dyadic annular-energy
argument makes the resulting error arbitrarily small along suitable
scales. Finally, separated annuli allow the selected sets to be
assembled into a single enumeration.
Section~\ref{sec:criterion} isolates the two hypotheses behind
this construction, and Section~\ref{sec:extensions} verifies them
for the two classes in the main theorem. Hyperbolic groups satisfy
them by the radial
multiplier theorem of Mei and de la Salle and the rapid-decay
inequality. For groups of polynomial growth, ball intersections
give the cutoffs and the volume bound gives the norm comparison.

\begin{remark}\label{rem:scope}
If $f$ is radial, then $c_g$ depends only on $|g|$, but $E_j$ need not
be a union of spheres. These convergence results rearrange
the terms $c_g\lambda_g$, not the blocks
$c_k\lambda(\chi_k)$. Neither radiality nor self-adjointness of $f$
is assumed. The assertions concern subsequences of partial sums;
they do not assert convergence of all partial sums or establish
a rearrangement theorem for whole radial shells.
\end{remark}

\section{Preliminaries}
The group algebra $\mathbb C[\F]$ has convolution and involution
\[
 (r*s)(g)=\sum_{h\in\F}r(gh^{-1})s(h),\qquad
 r^*(g)=\overline{r(g^{-1})}.
\]
Its left regular image is norm dense in $\Cr$. The operators
$(\lambda_g)_{g\in\F}$ are an orthonormal basis of $L^2(\M,\tau)$.
Consequently, for every $f\in L^2(\M,\tau)$,
\begin{equation}\label{eq:parseval}
 \sum_{g\in\F}|\widehat f(g)|^2=\norm{f}{2}^2<\infty,
 \qquad f=\sum_{g\in\F}\widehat f(g)\lambda_g\quad\hbox{in }L^2.
\end{equation}
The inequality $\norm{x}{p}\le\norm{x}{q}$ holds whenever
$1\le p\le q\le\infty$, since $\tau(1)=1$.
Thus $L^p(\widehat{\F})\subset L^2(\widehat{\F})$ when $p\ge2$.
Fourier polynomials are dense in $L^p(\M,\tau)$ for $1<p<\infty$.
Indeed, a continuous functional annihilating them is represented by
some $h\in L^{p'}(\M,\tau)\subset L^1(\M,\tau)$, where
$p'=p/(p-1)$. The normal functional $x\mapsto\tau(h^*x)$ then
vanishes on their ultraweak closure $\M$, so $h=0$; Hahn--Banach
gives the asserted density. We use the usual complex interpolation
identity for tracial noncommutative $L^p$ spaces; see \cite[Section~2]{PX}.
We use
\[
 \mathcal B_N=\{g:|g|\le N\},\qquad
 \mathcal A_N=\{g:N<|g|\le2N\}.
\]
These sets are finite. In the notation of the radial draft,
$\chi_k=\sum_{|g|=k}\delta_g$ and
$\norm{\lambda(\chi_k)}{2}^2=4\cdot3^{k-1}$ for $k\ge1$.
For a radial $f$ with coefficients $c_g=c_k$ on each sphere,
\begin{equation}\label{eq:radialenergy}
 \sum_{g\in\mathcal A_N}|c_g|^2
 =\sum_{k=N+1}^{2N}|c_k|^2\norm{\lambda(\chi_k)}{2}^2.
\end{equation}

We also recall the sufficient Schur factorization criterion. If a
kernel on a set $X$ has a Hilbert-space factorization
$K(x,y)=\langle A(x),B(y)\rangle$, with
$\sup_x\norm{A(x)}{}\sup_y\norm{B(y)}{}\le C$, its Schur multiplier
has completely bounded norm at most $C$; see \cite{HSS}.
Only this sufficient direction is used below.

\section{Smooth finite-support Fourier cutoffs}
We first bound the cutoffs on $\M$ and then obtain convergence in
finite $L^p$ spaces and in operator norm on $\Cr$.
Define the continuously differentiable function
\begin{equation}\label{eq:cutoff}
 F(t)=\begin{cases}
 1,&t\le1,\\
 2(t-1)^3-3(t-1)^2+1,&1\le t\le2,\\
 0,&t\ge2,
 \end{cases}
 \qquad p_N(k)=F(k/N).
\end{equation}
In particular $0\le p_N(k)\le1$ for $k\ge0$. The corresponding
Fourier multiplier is initially defined on Fourier polynomials by
\begin{equation}\label{eq:PN}
 P_N(\lambda_g)=p_N(|g|)\lambda_g.
\end{equation}

\begin{lemma}[Fourier estimate for second differences]\label{lem:fourier}
For $j\in\mathbb Z$, put
\[
 a_N(j)=F(j/N)-F((j+2)/N),\qquad
 w_N(\theta)=\sum_{j\in\mathbb Z}a_N(j)e^{-ij\theta}.
\]
With normalized measure $d\mu(\theta)=d\theta/(2\pi)$ on
$[-\pi,\pi]$, there is an absolute constant $C_0$ such that
\begin{equation}\label{eq:fourierL1}
 \norm{w_N}{L^1(\mu)}\le C_0/N\qquad(N\ge1).
\end{equation}
\end{lemma}
\begin{proof}
The derivative $F'$ is bounded by $3/2$ and is Lipschitz with
constant $6$. The sequence $a_N$ is supported in
$\{N-1,\ldots,2N-1\}$, so in particular it vanishes for negative
indices. Extending $F$ to the whole real line as in
\eqref{eq:cutoff} includes the boundary terms in the estimates
\[
 |a_N(j)|\le3/N,\qquad
 |a_N(j)-a_N(j-1)|\le12/N^2.
\]
For the second estimate one may write each difference using the
integral of $F'$ over an interval of length $2/N$ and translate that
interval by $1/N$. Thus
\begin{equation}\label{eq:diffbounds}
 \norm{a_N}{\ell^2(\mathbb Z)}\le3\sqrt2\,N^{-1/2},\qquad
 \norm{\Delta a_N}{\ell^2(\mathbb Z)}\le12\sqrt3\,N^{-3/2},
\end{equation}
where $\Delta a_N(j)=a_N(j)-a_N(j-1)$. Parseval gives
\[
 \norm{w_N}{2}=\norm{a_N}{\ell^2},\qquad
 \norm{(1-e^{-i\theta})w_N}{2}=\norm{\Delta a_N}{\ell^2}.
\]
On $|\theta|\le N^{-1}$, Cauchy--Schwarz bounds the integral of
$|w_N|$ by $(\pi N)^{-1/2}\norm{a_N}{\ell^2}\lesssim N^{-1}$.
On the complement, $|1-e^{-i\theta}|\ge2|\theta|/\pi$, so
\begin{align*}
 \int_{|\theta|>N^{-1}}|w_N|\,d\mu
 &\le\norm{\Delta a_N}{\ell^2}
 \left(\int_{|\theta|>N^{-1}}
          |1-e^{-i\theta}|^{-2}\,d\mu\right)^{1/2}\\
 &\le12\sqrt3\,N^{-3/2}(\pi N/4)^{1/2}\lesssim N^{-1}.
\end{align*}
Combining the regions proves \eqref{eq:fourierL1}.
\end{proof}

\begin{proposition}[Uniform cutoff bound]\label{prop:cutoff}
The maps $P_N$ extend to normal completely bounded maps on $\M$, and
\begin{equation}\label{eq:cb}
 \sup_{N\ge1}\norm{P_N}{\mathrm{cb}}\le C<\infty,\qquad C\ge1.
\end{equation}
For $2\le p<\infty$, they extend consistently to $L^p(\M,\tau)$ with
\begin{equation}\label{eq:lpbound}
 \sup_{N\ge1}\norm{P_N}{L^p\to L^p}\le C^{1-2/p}.
\end{equation}
For every $f\in L^p(\widehat{\F})$,
\begin{equation}\label{eq:cutoffconvergence}
 \norm{P_Nf-f}{p}\longrightarrow0.
\end{equation}
The cutoffs also converge in operator norm for $f\in\Cr$.
\end{proposition}
\begin{proof}
Consider the finite Hankel matrix $H_N=(a_N(i+j))_{i,j\ge0}$,
supported in its first $2N$ rows and columns. For
$v_\theta=(1,e^{i\theta},\ldots,e^{i(2N-1)\theta})^{\mathsf T}$,
Fourier inversion gives
\[
 H_N=\int_{-\pi}^{\pi}w_N(\theta)
                  v_\theta v_\theta^{\mathsf T}\,d\mu(\theta).
\]
The transpose, rather than the adjoint, produces the required
$e^{i(i+j)\theta}$ in each entry. This rank-one matrix has trace
norm $2N$. Hence Lemma~\ref{lem:fourier} implies
\begin{equation}\label{eq:hankel}
 \norm{H_N}{S_1}\le2N\norm{w_N}{1}\le2C_0.
\end{equation}

Equip $\F$ with the tree metric $d(x,y)=|xy^{-1}|$ and fix an end
of this tree. Let $(x_i)_{i\ge0}$ be the geodesic ray from $x=x_0$
toward this end. A singular-value decomposition factors $H_N$ as
\[
 a_N(i+j)=\langle\alpha_i,\beta_j\rangle,\qquad
 \sum_i\norm{\alpha_i}{}^2=\sum_j\norm{\beta_j}{}^2
 =\norm{H_N}{S_1}.
\]
All but finitely many vectors vanish. Set
\[
 A(x)=\sum_i\delta_{x_i}\otimes\alpha_i,
 \qquad B(y)=\sum_j\delta_{y_j}\otimes\beta_j.
\]
Vertices along each ray are distinct, giving
$\norm{A(x)}{}^2=\norm{B(y)}{}^2=\norm{H_N}{S_1}$.
If the rays first meet at $x_{i_0}=y_{j_0}$, then
$i_0+j_0=d(x,y)$, and their common vertices have indices
$(i_0+\ell,j_0+\ell)$, $\ell\ge0$. Therefore
\begin{align*}
 \langle A(x),B(y)\rangle
 &=\sum_{\ell\ge0}a_N(d(x,y)+2\ell)\\
 &=\sum_{\ell\ge0}\bigl[p_N(d(x,y)+2\ell)
                -p_N(d(x,y)+2\ell+2)\bigr]\\
 &=p_N(d(x,y)).
\end{align*}
This is a parity-preserving telescoping sum, terminating because
$p_N(k)=0$ for $k>2N$. The Schur factorization criterion and
\eqref{eq:hankel} give a completely bounded Schur multiplier of
norm at most $2C_0$. The matrix of $\lambda_g$ is supported on
pairs $x=gy$, for which $d(x,y)=|g|$. The Schur map is normal:
its factorization is $T\mapsto V_A^*(T\otimes I)V_B$, where
$V_A\delta_x=\delta_x\otimes A(x)$ and similarly for $V_B$.
On Fourier polynomials it agrees with the normal finite-rank map
\[
 P_Nx=\sum_{|g|\le2N}p_N(|g|)\tau(\lambda_g^*x)\lambda_g,
 \qquad x\in\M.
\]
Ultraweak density therefore identifies its restriction to $\M$
with this map, proving \eqref{eq:cb}.

Since $0\le p_N(k)\le1$, Parseval and the preceding estimate give
\[
 \norm{P_N}{L^2\to L^2}\le1,\qquad
 \norm{P_N}{\M\to\M}\le C.
\]
Interpolating these bounds proves \eqref{eq:lpbound}.
The extensions have the same finite Fourier formula, since each
coefficient functional is continuous on $L^p$.

Every Fourier polynomial $q$ is fixed by $P_N$ once $N$ exceeds
its largest word length. Given $f\in L^p(\widehat{\F})$,
approximate it in $L^p$ by such a $q$. Then
\[
 \norm{P_Nf-f}{p}\le(C^{1-2/p}+1)\norm{f-q}{p}
\]
for all sufficiently large $N$, proving
\eqref{eq:cutoffconvergence}. For $f\in\Cr$, the same density
argument in operator norm gives $\norm{P_Nf-f}{\infty}\to0$.
\end{proof}

\section[Random selection in finite Lp]{Random selection in finite $L^p$}
The next estimate replaces scalar-functional concentration. Its
constant depends only on the exponent, not on word lengths or the
number of selected terms. Although it is a special case of
noncommutative Khintchine estimates \cite{LP}, we include an
elementary proof for the even exponents needed here.

\begin{lemma}[Rademacher sums of unitaries]\label{lem:signs}
Let $(\mathcal N,\tau)$ be a finite von Neumann algebra with
$\tau(1)=1$. If $U_1,\ldots,U_s$ are unitaries, $d_1,\ldots,d_s$
are complex numbers, and $\varepsilon_i$ are independent Rademacher
variables with $\Prob(\varepsilon_i=1)=\Prob(\varepsilon_i=-1)=1/2$,
then, for every integer $\kappa\ge1$,
\begin{equation}\label{eq:rademacher}
 \E_\varepsilon\norm{\sum_{i=1}^s\varepsilon_i d_iU_i}{2\kappa}^{2\kappa}
 \le(2\kappa-1)!!\left(\sum_{i=1}^s|d_i|^2\right)^\kappa.
\end{equation}
\end{lemma}
\begin{proof}
Write $Y=\sum_i\varepsilon_i d_iU_i$ and expand
$\tau((Y^*Y)^\kappa)$. Sign averaging kills every term in which some
index occurs an odd number of times. Each remaining trace is that
of a product of unitaries and their adjoints, so its absolute
value is at most $1$. Taking absolute values of these terms gives
\[
 \E_\varepsilon\tau((Y^*Y)^\kappa)
 \le\E_\varepsilon\left(\sum_i\varepsilon_i|d_i|\right)^{2\kappa}.
\]
For $a_i=|d_i|$, the multinomial expansion and
$(2m)!\ge2^m m!$ yield
\begin{align*}
 \E_\varepsilon\left(\sum_i\varepsilon_i a_i\right)^{2\kappa}
 &=\sum_{m_1+\cdots+m_s=\kappa}
       \frac{(2\kappa)!}{\prod_i(2m_i)!}\prod_i a_i^{2m_i}\\
 &\le\frac{(2\kappa)!}{2^\kappa}
       \sum_{m_1+\cdots+m_s=\kappa}\frac{\prod_i a_i^{2m_i}}{\prod_i m_i!}\\
 &=\frac{(2\kappa)!}{2^\kappa \kappa!}\left(\sum_i a_i^2\right)^\kappa.
\end{align*}
Since $(2\kappa-1)!!=(2\kappa)!/(2^\kappa \kappa!)$, this proves the claim.
\end{proof}

\begin{lemma}[Bernoulli rounding]\label{lem:rounding}
Let $G$ be a discrete group, let $A\subset G$ be finite, let $c_g\in\mathbb C$ and
$0\le p_g\le1$ for $g\in A$, and let $X_g$ be independent
Bernoulli variables with means $p_g$. Put
\[
 Z_\omega=\sum_{g\in A}(X_g(\omega)-p_g)c_g\lambda_g,
 \qquad V=\sum_{g\in A}|c_g|^2.
\]
For every integer $\kappa\ge1$,
\begin{equation}\label{eq:roundmoment}
 \E\norm{Z_\omega}{2\kappa}^{2\kappa}\le(2\kappa-1)!!V^\kappa.
\end{equation}
In particular, there is a subset $D\subset A$ such that
\begin{equation}\label{eq:roundselection}
 \norm{\sum_{g\in D}c_g\lambda_g-
                  \sum_{g\in A}p_gc_g\lambda_g}{2\kappa}
 \le K_{2\kappa}\sqrt V,
 \qquad K_{2\kappa}=((2\kappa-1)!!)^{1/(2\kappa)}\le\sqrt{2\kappa}.
\end{equation}
For every $t>0$ we also have
\begin{equation}\label{eq:markov}
 \Prob\{\norm{Z_\omega}{2\kappa}>t\}
 \le\frac{(2\kappa-1)!!V^\kappa}{t^{2\kappa}}.
\end{equation}
\end{lemma}
\begin{proof}
Let $(X'_g)$ be an independent copy. Conditional expectation
over this copy gives $Z_\omega=\E_{X'}\sum_g(X_g-X'_g)c_g\lambda_g$.
Convexity of $y\mapsto\norm{y}{2\kappa}^{2\kappa}$ implies
\[
 \E_X\norm{Z_\omega}{2\kappa}^{2\kappa}
 \le\E_{X,X'}\norm{\sum_g(X_g-X'_g)c_g\lambda_g}{2\kappa}^{2\kappa}.
\]
The variables $X_g-X'_g$ are independent and symmetric, with
values $0,1,-1$. Their joint distribution is that of
$\varepsilon_g\delta_g$, where the $\varepsilon_g$ are mutually
independent Rademacher variables, independent of all the independent
Bernoulli variables $\delta_g$ with
$\Prob(\delta_g=1)=2p_g(1-p_g)$.
Condition on the $\delta_g$ and apply Lemma~\ref{lem:signs}:
\[
 \E_\varepsilon\norm{\sum_g\varepsilon_g\delta_gc_g\lambda_g}{2\kappa}^{2\kappa}
 \le(2\kappa-1)!!\left(\sum_g\delta_g^2|c_g|^2\right)^\kappa
 \le(2\kappa-1)!!V^\kappa.
\]
This proves \eqref{eq:roundmoment}. At least one outcome has its
$2\kappa$-th power norm no larger than its expectation; its selected
indices give $D$. The estimate for $K_{2\kappa}$ follows by bounding
each factor in $(2\kappa-1)!!$ by $2\kappa$. Finally,
\eqref{eq:markov} is Markov's inequality applied to the
nonnegative variable $\norm{Z_\omega}{2\kappa}^{2\kappa}$.
\end{proof}

\begin{remark}
The quantity $V$ is the coefficient energy, not the actual
Bernoulli variance $\sum_g p_g(1-p_g)|c_g|^2$. The proof uses
randomness before taking absolute bounds: sign averaging is
performed inside the trace expansion. No self-adjointness of
$Z_\omega$ is required.
\end{remark}

\begin{corollary}[Selection on a transition annulus]\label{cor:annulus}
For $f\in L^2(\widehat{\F})$, $N\ge1$, and an integer $\kappa\ge1$,
there is a subset
$D\subset\mathcal A_N$ such that
\begin{equation}\label{eq:annulus}
 \norm{\sum_{g\in\mathcal B_N\cup D}\widehat f(g)\lambda_g-P_Nf}{2\kappa}
 \le K_{2\kappa}\left(\sum_{g\in\mathcal A_N}|\widehat f(g)|^2\right)^{1/2}.
\end{equation}
\end{corollary}
\begin{proof}
Apply Lemma~\ref{lem:rounding} with $A=\mathcal A_N$ and
$p_g=p_N(|g|)$. The terms in $\mathcal B_N$ have weight $1$
and cancel from the difference.
\end{proof}

\section{Operator-norm control of the rounding error}
All logarithms in this section are natural. For a Fourier polynomial
$y$, put
\[
 \Pi_k y=\sum_{|g|=k}\widehat y(g)\lambda_g.
\]
We say that $y$ has degree at most $L$ if its Fourier coefficients
vanish for $|g|>L$. Haagerup's inequality
\cite[Lemma~1.4]{Ha79} gives
$\norm{\Pi_k y}{\infty}\le(k+1)\norm{\Pi_k y}{2}$. Thus, for
$y$ of degree at most $L$, Cauchy--Schwarz and orthogonality give
\begin{equation}\label{eq:ballRD}
 \norm{y}{\infty}
 \le\sum_{k=0}^L(k+1)\norm{\Pi_k y}{2}
 \le\left(\sum_{k=0}^L(k+1)^2\right)^{1/2}\norm{y}{2}
 \le(L+1)^{3/2}\norm{y}{2}.
\end{equation}

\begin{lemma}[Comparison with a high moment]\label{lem:highmoment}
If $z$ is a Fourier polynomial of degree at most $D\ge1$, then
for every integer $m\ge1$,
\begin{equation}\label{eq:highmoment}
 \norm{z}{\infty}\le(2mD+1)^{3/(4m)}\norm{z}{4m}.
\end{equation}
In particular, for $m=\lceil\log(D+1)\rceil$,
\begin{equation}\label{eq:logmoment}
 \norm{z}{\infty}\le e^{3/2}\norm{z}{4m}.
\end{equation}
\end{lemma}
\begin{proof}
The polynomial $y=(z^*z)^m$ has degree at most $2mD$.
Applying \eqref{eq:ballRD} yields
\[
 \norm{z}{\infty}^{2m}
 =\norm{(z^*z)^m}{\infty}
 \le(2mD+1)^{3/2}\norm{(z^*z)^m}{2}
 =(2mD+1)^{3/2}\norm{z}{4m}^{2m}.
\]
Taking the $2m$-th root proves \eqref{eq:highmoment}.
For the stated choice of $m$, we have $m\ge1$, $D+1\le e^m$,
and $2m\le e^m$. Consequently,
\[
 2mD+1\le2m(D+1)\le2me^m\le e^{2m},
\]
which proves \eqref{eq:logmoment}. The argument does not require
$z$ to be self-adjoint.
\end{proof}

\begin{corollary}[Operator-norm rounding]\label{cor:operatorround}
Let $A\subset\mathcal B_D$ be finite, where $D\ge1$ is an integer.
For any $c_g\in\mathbb C$ and $0\le p_g\le1$, $g\in A$, there
is a subset $S\subset A$ such that
\begin{equation}\label{eq:operatorround}
 \norm{\sum_{g\in S}c_g\lambda_g-
              \sum_{g\in A}p_gc_g\lambda_g}{\infty}
 \le9\left((1+\log(D+1))\sum_{g\in A}|c_g|^2\right)^{1/2}.
\end{equation}
\end{corollary}
\begin{proof}
Set $m=\lceil\log(D+1)\rceil$ and $V=\sum_{g\in A}|c_g|^2$.
Apply Lemma~\ref{lem:rounding} with $\kappa=2m$. It supplies $S$ for
which the difference $z$ in \eqref{eq:operatorround} satisfies
\[
 \norm{z}{4m}\le K_{4m}\sqrt V\le2\sqrt{mV}.
\]
This difference has degree at most $D$, so
Lemma~\ref{lem:highmoment} gives
\[
 \norm{z}{\infty}\le2e^{3/2}\sqrt{mV}
 \le9\sqrt{(1+\log(D+1))V},
\]
using $2e^{3/2}<9$ and $m\le1+\log(D+1)$.
\end{proof}

\begin{lemma}[Small dyadic annular energies]\label{lem:dyadic}
For $f\in L^2(\widehat{\F})$, define
\[
 V_N=\sum_{N<|g|\le2N}|\widehat f(g)|^2.
\]
Then
\begin{equation}\label{eq:dyadic}
 \sum_{k=0}^\infty V_{2^k}\le\norm{f}{2}^2,
 \qquad \liminf_{k\to\infty}(k+1)V_{2^k}=0.
\end{equation}
In particular,
\begin{equation}\label{eq:logenergy}
 \liminf_{k\to\infty}
       (1+\log(2^{k+1}+1))V_{2^k}=0.
\end{equation}
\end{lemma}
\begin{proof}
The annuli $\mathcal A_{2^k}$ are disjoint, so Parseval proves
summability. If the first limit inferior were positive, then
$V_{2^k}\ge c/(k+1)$ for some $c>0$ and all sufficiently large
$k$, contradicting divergence of the harmonic series.
Finally, $1+\log(2^{k+1}+1)\le3(k+1)$ for $k\ge0$, which gives
\eqref{eq:logenergy}.
\end{proof}

\section{Construction of the rearrangements}
\begin{proof}[Proof of Theorem~\ref{thm:main}]
Fix $2\le p<\infty$ and $f\in L^p(\widehat{\F})$.
Set $\kappa=\lceil p/2\rceil$, write $c_g=\widehat f(g)$, and define
\[
 T_N=\sum_{|g|>N}|c_g|^2.
\]
By \eqref{eq:parseval}, $T_N\to0$. Proposition~\ref{prop:cutoff}
gives $\norm{P_Nf-f}{p}\to0$. Starting with $n_0=0$,
choose integers $n_j$ recursively so that
\begin{equation}\label{eq:choose}
 n_j>2n_{j-1},\qquad
 \norm{P_{n_j}f-f}{p}\le2^{-j},\qquad
 K_{2\kappa}\sqrt{T_{n_j}}\le2^{-j}.
\end{equation}
For each fixed $j$, these conditions hold for all sufficiently
large $n_j$, so the recursive choice is possible.

Apply Corollary~\ref{cor:annulus} with $N=n_j$ and this fixed $\kappa$.
It gives $D_j\subset\mathcal A_{n_j}$. Put
\[
 E_j=\mathcal B_{n_j}\cup D_j,\qquad
 F_j=\sum_{g\in E_j}c_g\lambda_g.
\]
Since the annular coefficient energy is at most $T_{n_j}$,
\eqref{eq:choose} and \eqref{eq:annulus} show that
\[
 \norm{F_j-f}{p}
 \le\norm{F_j-P_{n_j}f}{2\kappa}+\norm{P_{n_j}f-f}{p}
 \le2^{-j}+2^{-j}=2^{1-j}.
\]
Here we used $p\le2\kappa$ only on the rounding error, which is a Fourier
polynomial; no assumption $f\in L^{2\kappa}$ is needed.
This proves \eqref{eq:rate}. Norm monotonicity gives convergence
in every $L^q$ with $1\le q\le p$.

By construction,
\[
 E_j\subset\mathcal B_{2n_j}\subsetneq
 \mathcal B_{n_{j+1}}\subset E_{j+1}.
\]
They exhaust $\F$ because $n_j\to\infty$ and
$\mathcal B_{n_j}\subset E_j$. Enumerate first $E_1$, then
$E_2\setminus E_1$, then $E_3\setminus E_2$, and so on,
in any order within each finite set. This lists each group
element exactly once, including those whose Fourier coefficient
is zero. With $M_j=|E_j|$, its partial sum at $M_j$ equals $F_j$.
The integers $M_j$ are strictly increasing, and \eqref{eq:main}
follows.
\end{proof}

\begin{proof}[Proof of Theorem~\ref{thm:operator}]
Let $f\in\Cr$ and $c_g=\widehat f(g)$. Since $f\in L^2$, use
the annular energies $V_N$ from Lemma~\ref{lem:dyadic}.
Proposition~\ref{prop:cutoff} gives
$\norm{P_Nf-f}{\infty}\to0$. Recursively choose dyadic integers
$n_j=2^{k_j}$ such that
\begin{equation}\label{eq:operatorchoose}
 \begin{gathered}
 n_j>2n_{j-1},\qquad \norm{P_{n_j}f-f}{\infty}\le2^{-j},\\
 9\sqrt{(1+\log(2n_j+1))V_{n_j}}\le2^{-j},
 \qquad n_0=0.
 \end{gathered}
\end{equation}
The cutoff error is small for all sufficiently large radii, while
\eqref{eq:logenergy} supplies arbitrarily large dyadic radii
satisfying the last error bound. Thus all conditions can be met
at every stage.

Apply Corollary~\ref{cor:operatorround} on
$A=\mathcal A_{n_j}\subset\mathcal B_{2n_j}$ with
$p_g=p_{n_j}(|g|)$. It gives $D_j\subset\mathcal A_{n_j}$.
Set $E_j=\mathcal B_{n_j}\cup D_j$ and
$F_j=\sum_{g\in E_j}c_g\lambda_g$. The terms of length at most
$n_j$ cancel in $F_j-P_{n_j}f$, so
\[
 \norm{F_j-f}{\infty}
 \le9\sqrt{(1+\log(2n_j+1))V_{n_j}}
       +\norm{P_{n_j}f-f}{\infty}
 \le2^{1-j}.
\]
As in the proof of Theorem~\ref{thm:main}, the separated radii
make the sets $E_j$ strictly nested and exhaustive, with
\eqref{eq:sets}. Enumerating their successive differences and
putting $M_j=|E_j|$ proves \eqref{eq:operator} and
\eqref{eq:operatorrate}. Norm monotonicity gives convergence in
every finite $L^q$ along the same subsequence.
\end{proof}

\begin{proof}[Proof of Corollary~\ref{cor:simultaneous}]
Suppose $f\in\bigcap_{2\le q<\infty}L^q(\widehat{\F})$.
Repeat the construction with exponent $2j$ at stage $j$, choosing
\[
 n_j>2n_{j-1},\qquad
 \norm{P_{n_j}f-f}{2j}\le2^{-j},\qquad
 K_{2j}\sqrt{T_{n_j}}\le2^{-j}.
\]
For each $j$ this is possible by Proposition~\ref{prop:cutoff} and
$T_N\to0$. Corollary~\ref{cor:annulus} with $\kappa=j$ gives nested
sets satisfying \eqref{eq:simrate}. For every fixed finite $q$,
eventually $q\le2j$, and then
$\norm{F_j-f}{q}\le\norm{F_j-f}{2j}\le2^{1-j}$.
The same enumeration of the successive differences of these sets
therefore works for all finite $q$. Finally, $\tau(1)=1$ implies
$\M\subset\bigcap_{2\le q<\infty}L^q(\widehat{\F})$.
\end{proof}

\section{A general convergence criterion}\label{sec:criterion}
We now separate the geometric input from the selection argument.
Throughout this section $\Gamma$ is an infinite finitely generated
group with a fixed finite symmetric generating set, and
$B_N=\{g\in\Gamma:|g|\le N\}$. The trace, Fourier coefficients,
Parseval identity, density of Fourier polynomials, and norm
monotonicity are the same as in Section~2, with $\Gamma$ replacing
$\F$. Lemma~\ref{lem:rounding} already holds for every discrete group.

The required hypotheses are as follows. There are constants
$C_0,C_1\ge1$, $s\ge0$, an integer $a\ge2$, and Fourier multipliers
$Q_N(\lambda_g)=q_N(g)\lambda_g$ such that
\begin{equation}\label{eq:generalcutoffs}
 \begin{gathered}
 0\le q_N(g)\le1,\qquad q_N(g)=1\quad(g\in B_N),\\
 \operatorname{supp}q_N\subset B_{aN},\qquad
 \sup_{N\ge1}\norm{Q_N}{\mathrm{cb}}\le C_0,
 \end{gathered}
\end{equation}
where the completely bounded norms are on $L(\Gamma)$, and
\begin{equation}\label{eq:generalRD}
 \norm{y}{\infty}\le C_1(R+1)^s\norm{y}{2}
 \quad\text{if }\operatorname{supp}\widehat y\subset B_R.
\end{equation}
The maps in \eqref{eq:generalcutoffs} are understood as normal
finite-rank Fourier multipliers. The second condition is a polynomial
rapid-decay bound; no special value of $s$ is required.

\begin{proposition}[General convergence criterion]\label{prop:criterion}
If \eqref{eq:generalcutoffs} and \eqref{eq:generalRD} hold, then all
conclusions of Theorem~\ref{thm:groups} hold for $\Gamma$, with this
value of $a$. The finite-$L^p$ and simultaneous finite-$L^q$
conclusions require only \eqref{eq:generalcutoffs}.
\end{proposition}
\begin{proof}
Since $|q_N(g)|\le1$, Parseval gives $\norm{Q_N}{L^2\to L^2}\le1$.
Interpolation and density, exactly as in
Proposition~\ref{prop:cutoff}, give
\begin{equation}\label{eq:generalapprox}
 \begin{aligned}
 \norm{Q_Nf-f}{p}&\longrightarrow0
      &&(f\in L^p(\widehat\Gamma),\ 2\le p<\infty),\\
 \norm{Q_Nf-f}{\infty}&\longrightarrow0
      &&(f\in C_r^*(\Gamma)).
 \end{aligned}
\end{equation}
Indeed, the $L^p$ norms of $Q_N$ are bounded by $C_0^{1-2/p}$,
and $Q_N$ fixes every Fourier polynomial for all sufficiently large
$N$. Write $c_g=\widehat f(g)$ and put
\[
 A_N=B_{aN}\setminus B_N,\qquad
 V_N=\sum_{g\in A_N}|c_g|^2,\qquad
 T_N=\sum_{|g|>N}|c_g|^2.
\]

For fixed finite $p$, set $\kappa=\lceil p/2\rceil$. Choose
$n_j>a n_{j-1}$ so that
\[
 \norm{Q_{n_j}f-f}{p}\le2^{-j},\qquad
 K_{2\kappa}\sqrt{T_{n_j}}\le2^{-j}.
\]
Lemma~\ref{lem:rounding}, applied to $A_{n_j}$ with probabilities
$q_{n_j}(g)$, gives $D_j\subset A_{n_j}$ for which, setting
$E_j=B_{n_j}\cup D_j$ and $F_j=\sum_{g\in E_j}c_g\lambda_g$,
\[
 \norm{F_j-Q_{n_j}f}{p}
 \le\norm{F_j-Q_{n_j}f}{2\kappa}
 \le K_{2\kappa}\sqrt{V_{n_j}}\le2^{-j}.
\]
This proves $\norm{F_j-f}{p}\le2^{1-j}$. As before, the
$L^{2\kappa}$ estimate is used only on the polynomial rounding error.
If $f$ belongs to every finite $L^q$, repeat the choices with
$p=2j$ and $\kappa=j$ at stage $j$. The resulting estimate
$\norm{F_j-f}{2j}\le2^{1-j}$ gives simultaneous convergence.

For the endpoint, first extend the high-moment argument.
If $z$ has degree at most $D\ge1$, then $(z^*z)^m$ has degree
at most $2mD$. Applying \eqref{eq:generalRD} to this power gives
\begin{equation}\label{eq:generalhighmoment}
 \norm{z}{\infty}
 \le C_1^{1/(2m)}(2mD+1)^{s/(2m)}\norm{z}{4m}.
\end{equation}
For $m=\lceil\log(D+1)\rceil$, the estimate
$2mD+1\le e^{2m}$ from Lemma~\ref{lem:highmoment} shows that
the prefactor is at most $\sqrt{C_1}e^s$.
Combining this with Lemma~\ref{lem:rounding}, with $\kappa=2m$,
shows that for any finite $A\subset B_D$ and any probabilities
$p_g\in[0,1]$ there is a subset $S\subset A$ with
\begin{equation}\label{eq:generalrounding}
 \norm{\sum_{g\in S}c_g\lambda_g-
               \sum_{g\in A}p_gc_g\lambda_g}{\infty}
 \le C_*\sqrt{(1+\log(D+1))\sum_{g\in A}|c_g|^2},
 \qquad C_*=2\sqrt{C_1}e^s.
\end{equation}
In particular this applies to $A_N$ with $D=aN$.

The annuli $A_{a^k}=B_{a^{k+1}}\setminus B_{a^k}$ are disjoint.
Consequently,
\[
 \sum_{k\ge0}V_{a^k}\le\norm{f}{2}^2,
 \qquad \liminf_{k\to\infty}(k+1)V_{a^k}=0.
\]
The second assertion follows by the harmonic-series argument in
Lemma~\ref{lem:dyadic}. Since
$1+\log(a^{k+1}+1)\le C_a(k+1)$, it also gives
\begin{equation}\label{eq:generalenergy}
 \liminf_{k\to\infty}
   (1+\log(a^{k+1}+1))V_{a^k}=0.
\end{equation}
For $f\in C_r^*(\Gamma)$, choose $n_j=a^{k_j}$ recursively with
\[
 \begin{gathered}
 n_j>a n_{j-1},\qquad \norm{Q_{n_j}f-f}{\infty}\le2^{-j},\\
 C_*\sqrt{(1+\log(a n_j+1))V_{n_j}}\le2^{-j}.
 \end{gathered}
\]
The approximation error is small at all sufficiently large radii,
and \eqref{eq:generalenergy} supplies arbitrarily large good
geometric radii for the rounding error. Apply
\eqref{eq:generalrounding} to obtain $D_j\subset A_{n_j}$ and
again set $E_j=B_{n_j}\cup D_j$. Then
\[
 \norm{F_j-f}{\infty}
 \le\norm{F_j-Q_{n_j}f}{\infty}
       +\norm{Q_{n_j}f-f}{\infty}\le2^{1-j}.
\]

In each construction, start with $n_0=0$. The separated radii give
$E_j\subset B_{a n_j}\subsetneq B_{n_{j+1}}\subset E_{j+1}$,
and the sets exhaust $\Gamma$. Enumerating their successive
differences, including zero coefficients, and setting $M_j=|E_j|$
proves the required rearrangement statements and
\eqref{eq:groupssets}. Norm monotonicity gives all the stated
lower-exponent consequences.
\end{proof}

\section{Hyperbolic groups and groups of polynomial growth}\label{sec:extensions}
We apply Proposition~\ref{prop:criterion} to the two classes appearing
in Theorem~\ref{thm:groups}.

\subsection{Hyperbolic groups}
We verify \eqref{eq:generalcutoffs} with $a=2$ and
$q_N(g)=F(|g|/N)$, using the same $C^{1,1}$ function $F$ as in
\eqref{eq:cutoff}. The hyperbolic radial multiplier theorem
\cite[Corollary~2.8]{MdS} bounds the completely bounded norm of
this multiplier by a group-dependent constant times the trace
norm of the \emph{first-difference} Hankel matrix
\[
 H_N^{(1)}=
 \bigl(F((i+j)/N)-F((i+j+1)/N)\bigr)_{i,j\ge0}.
\]
The limiting constant in that theorem is zero because the symbol
has finite support. Notice that this matrix uses a shift by $1$,
whereas the tree argument of Section~3 used a shift by $2$.

For completeness, its trace norm is uniformly bounded by the
same elementary Fourier calculation. Set
\[
 b_N(j)=F(j/N)-F((j+1)/N),\qquad
 v_N(\theta)=\sum_{j\in\mathbb Z}b_N(j)e^{-ij\theta}.
\]
The sequence $b_N$ is supported in $\{N,\ldots,2N-1\}$, and
the bounds on $F'$ and its Lipschitz constant give
\[
 \norm{b_N}{\ell^2}\le\tfrac32N^{-1/2},\qquad
 \norm{\Delta b_N}{\ell^2}\le6\sqrt2\,N^{-3/2}.
\]
Splitting the integral at $|\theta|=N^{-1}$, exactly as in
Lemma~\ref{lem:fourier}, yields $\norm{v_N}{L^1(\mu)}\lesssim N^{-1}$.
The rank-one integral representation from the proof of
Proposition~\ref{prop:cutoff}, now with $v_N$ in place of $w_N$,
therefore gives
\[
 \norm{H_N^{(1)}}{S_1}\le2N\norm{v_N}{L^1(\mu)}\lesssim1.
\]
Applying \cite[Corollary~2.8]{MdS} proves
\eqref{eq:generalcutoffs}. The relevant graph metric may be
taken as $d(x,y)=|xy^{-1}|$, which is isometric under inversion
to the usual Cayley graph metric and matches the matrix support
of the left regular representation.

The rapid-decay theorem for hyperbolic groups gives
\begin{equation}\label{eq:hyperbolicRD}
 \norm{y}{\infty}\le C_\Gamma(R+1)^{3/2}\norm{y}{2}
 \quad(\operatorname{supp}\widehat y\subset B_R);
\end{equation}
see \cite{dH} and the formulation \textup{$(H_\bullet)$} in
\cite{Nica}. That formulation is stated for non-elementary
hyperbolic groups. Infinite elementary hyperbolic groups are
virtually cyclic, and their linear ball growth gives a stronger
bound directly by Cauchy--Schwarz. Thus
\eqref{eq:generalRD} holds for every infinite hyperbolic group.
Proposition~\ref{prop:criterion} proves all assertions of
Theorem~\ref{thm:groups} in this case.

\subsection{Groups of polynomial growth}
Suppose that, for some $C_v\ge1$ and $d\ge0$,
\begin{equation}\label{eq:polynomialvolume}
 |B_R|\le C_v(R+1)^d\qquad(R\ge0).
\end{equation}
Word balls in a finitely generated group of polynomial growth
also satisfy uniform doubling:
\begin{equation}\label{eq:doubling}
 |B_{2N}|\le C_d|B_N|\qquad(N\ge1).
\end{equation}
See \cite[p.~54, following Corollary~10]{Tessera} for this
standard consequence of the structure of polynomial-growth groups.
Define
\begin{equation}\label{eq:polynomialcutoff}
 q_N(g)=\frac{|gB_N\cap B_{2N}|}{|B_N|}.
\end{equation}
Clearly $0\le q_N(g)\le1$. If $|g|\le N$, then
$gB_N\subset B_{2N}$, so $q_N(g)=1$. If $q_N(g)\ne0$, write
$ga=b$ with $a\in B_N$ and $b\in B_{2N}$; then
$|g|=|ba^{-1}|\le3N$. Hence this cutoff has the plateau and
support properties in \eqref{eq:generalcutoffs} with $a=3$.

To verify its multiplier bound directly, put
$\xi_N=|B_N|^{-1/2}\mathbf1_{B_N}$ and
$\eta_N=|B_N|^{-1/2}\mathbf1_{B_{2N}}$ in $\ell^2(\Gamma)$.
With our convention on inner products,
$q_N(g)=\langle\eta_N,\lambda_g\xi_N\rangle$. The kernel
matching a Fourier multiplier on the left regular representation
has the factorization
\[
 q_N(xy^{-1})
 =\langle\lambda_{x^{-1}}\eta_N,
                 \lambda_{y^{-1}}\xi_N\rangle.
\]
The Schur factorization criterion therefore gives a normal
completely bounded Fourier multiplier $Q_N$ with
\begin{equation}\label{eq:polynomialcb}
 \norm{Q_N}{\mathrm{cb}}
 \le\norm{\eta_N}{2}\norm{\xi_N}{2}
 =\sqrt{\frac{|B_{2N}|}{|B_N|}}\le\sqrt{C_d}.
\end{equation}
This argument needs neither positive definiteness nor radiality
of $q_N$. It proves all of \eqref{eq:generalcutoffs}.

Finally, for a polynomial supported in $B_R$, the unitary norm
of each $\lambda_g$, Cauchy--Schwarz, and
\eqref{eq:polynomialvolume} give
\begin{equation}\label{eq:polynomialRD}
 \norm{y}{\infty}\le\sum_{g\in B_R}|\widehat y(g)|
 \le\sqrt{|B_R|}\norm{y}{2}
 \le\sqrt{C_v}(R+1)^{d/2}\norm{y}{2}.
\end{equation}
Thus \eqref{eq:generalRD} holds with $s=d/2$.
Proposition~\ref{prop:criterion}, using the disjoint annuli
$B_{3^{k+1}}\setminus B_{3^k}$, proves the polynomial-growth
case of Theorem~\ref{thm:groups}. This completes its proof and,
consequently, that of Corollary~\ref{cor:conjectures}.
\medskip

 {\bf Acknowledgements. } {The authors are grateful to Ben Hayes for helpful discussions and for his suggestion to study  R\'ev\'esz's convergence theory on free groups}.


\begin{thebibliography}{14}
\bibitem{An23} Andreou, D.:
Crossed products of dual operator spaces and a characterization of groups with the approximation property. J. Oper. Theory (2023). doi: http://dx.doi.org/10.7900/jot.2021aug22.2341

\bibitem{BF06}
M.~Bo\.{z}ejko and G.~Fendler,
\emph{A note on certain partial sum operators},
in \emph{Quantum Probability}, Banach Center Publ. \textbf{73},
Polish Academy of Sciences, Institute of Mathematics, Warsaw, 2006,
117--125.

\bibitem{HHP19}
K.~Hamm, B.~Hayes, and A.~Petrosyan,
\emph{Rearranged Fourier Series and Generalizations to Non-Commutative Groups},
2019 13th International Conference on Sampling Theory and Applications
(SampTA), 2019, 1--4.
\href{https://doi.org/10.1109/SampTA45681.2019.9030904}{doi:10.1109/SampTA45681.2019.9030904}.

\bibitem{Hayesetal1}
K.~Hamm, B.~Hayes, and A.~Petrosyan,
\emph{An operator theoretic approach to the convergence of rearranged
Fourier series},
J. Anal. Math. (2021).
\href{https://doi.org/10.1007/s11854-021-0161-8}{doi:10.1007/s11854-021-0161-8}.

\bibitem{MdS}
T.~Mei and M.~de la Salle,
\emph{Complete boundedness of heat semigroups on the von Neumann algebra
of hyperbolic groups},
Transactions of the American Mathematical Society \textbf{369} (2017),
no.~8, 5601--5622.
\href{https://arxiv.org/abs/1405.5178}{arXiv:1405.5178}.

\bibitem{dH}
P.~de la Harpe,
\emph{Groupes hyperboliques, alg\`ebres d'op\'erateurs et un th\'eor\`eme
de Jolissaint},
C.~R. Acad. Sci. Paris S\'er.~I Math. \textbf{307} (1988), no.~14, 771--774.

\bibitem{Nica}
B.~Nica,
\emph{On operator norms for hyperbolic groups},
Journal of Topology and Analysis \textbf{9} (2017), no.~2, 291--296.
\href{https://arxiv.org/abs/1509.08317}{arXiv:1509.08317}.

\bibitem{Tessera}
R.~Tessera,
\emph{Volume of spheres in doubling metric measured spaces and in groups
of polynomial growth},
Bulletin de la Soci\'et\'e Math\'ematique de France \textbf{135} (2007),
no.~1, 47--64.
\href{https://doi.org/10.24033/bsmf.2525}{doi:10.24033/bsmf.2525}.

\bibitem{Ha79}
U.~Haagerup,
\emph{An example of a non nuclear $C^*$-algebra, which has the metric
approximation property},
Inventiones Mathematicae \textbf{50} (1979), 279--293.

\bibitem{PX}
G.~Pisier and Q.~Xu,
\emph{Non-commutative $L^p$-spaces},
in \emph{Handbook of the Geometry of Banach Spaces}, Vol.~2,
North-Holland, 2003, 1459--1517.

\bibitem{HSS}
U.~Haagerup, T.~Steenstrup, and R.~Szwarc,
\emph{Schur multipliers and spherical functions on homogeneous trees},
\href{https://arxiv.org/abs/0908.4424}{arXiv:0908.4424}.

\bibitem{LP}
F.~Lust-Piquard and G.~Pisier,
\emph{Non commutative Khintchine and Paley inequalities},
Arkiv f\"or Matematik \textbf{29} (1991), 241--260.


\bibitem{Ozawa}
N.~Ozawa,
\emph{Weak amenability of hyperbolic groups},
Groups Geom. Dyn. \textbf{2} (2008), 271--280.

\bibitem{Re90}
S.~G.~R\'ev\'esz,
\emph{Rearrangements of Fourier series},
J. Approx. Theory (1990).
\href{https://doi.org/10.1016/0021-9045(90)90076-3}{doi:10.1016/0021-9045(90)90076-3}.
\end{thebibliography}
\end{document}